\documentclass{article}
\usepackage[utf8]{inputenc}

\usepackage{bbm}
\usepackage{comment}
\usepackage{verbatim}
\usepackage{mathrsfs}
\usepackage{mathtools}
\usepackage{latexsym}
\usepackage{epsfig}
\usepackage{amsmath}
\usepackage[usenames,dvipsnames]{xcolor}
\usepackage{bbm}
\usepackage{graphics}
\usepackage[utf8]{inputenc}
\usepackage{amssymb}
\usepackage{amsthm}
\usepackage{amsopn}
\usepackage{amscd}
\usepackage[all,knot]{xy}
\xyoption{all}
\usepackage{rotating}
\usepackage{tikz}
\usetikzlibrary{calc}
\usepgflibrary{shapes.geometric}
\usepgflibrary{shapes.misc}
\usetikzlibrary{positioning}
\usetikzlibrary{decorations}
\usetikzlibrary{decorations.pathreplacing}
\usepackage{tikz}
\usepackage{tikz-cd}
\mathtoolsset{showonlyrefs}

\usepackage[hidelinks]{hyperref}

\theoremstyle{plain}
\newtheorem{thm}{Theorem}[section]
\newtheorem{lem}[thm]{Lemma}
\newtheorem{prop}[thm]{Proposition}
\newtheorem{cor}[thm]{Corollary}

\newtheorem*{thm*}{Theorem}
\newtheorem*{lem*}{Lemma}
\newtheorem*{prop*}{Proposition}
\newtheorem*{cor*}{Corollary}

\theoremstyle{definition}
\newtheorem{defn}[thm]{Definition}
\newtheorem*{defn*}{Definition}

\newtheorem{ex*}{Example}
\newtheorem{rem}[thm]{Remark}
\newtheorem*{rem*}{Remark}

\newtheorem{notation}[thm]{Notation}{}
{}
{}
{}
\newtheorem*{assump*}{Assumption}

\theoremstyle{remark}
{}
{}
{}

{}

\def\to{\longrightarrow} 
\renewcommand{\xrightarrow}[1]{\stackrel{#1}{\to}}

\DeclareUnicodeCharacter{221E}{$\infty$}

\def\sfT{\mathsf{T}}

\def\mcE{\mathcal{E}}

\def\mcI{\mathcal{I}}
\def\mcJ{\mathcal{J}}

\def\mcM{\mathcal{M}}

\def\mcP{\mathcal{P}}
\def\mcQ{\mathcal{Q}}

\def\mcS{\mathcal{S}}
\def\mcT{\mathcal{T}}

\def\mcX{\mathcal{X}}

\def\spe{\mathrm{sp}}

\def\unit{\mathbf{1}}

\DeclareMathOperator{\Spc}{Spc}

\DeclareMathOperator{\supp}{supp}

\DeclareMathOperator{\thick}{thick}

\definecolor{internationalkleinblue}{rgb}{0.0, 0.18, 0.65}

\definecolor{darkorange}{rgb}{1.00, 0.58, 0.00}

\newcommand\imCMsym[4][\mathord]{%
	\DeclareFontFamily{U} {#2}{}
	\DeclareFontShape{U}{#2}{m}{n}{
		<-6> #25
		<6-7> #26
		<7-8> #27
		<8-9> #28
		<9-10> #29
		<10-12> #210
		<12-> #212}{}
	\DeclareSymbolFont{CM#2} {U} {#2}{m}{n}
	\DeclareMathSymbol{#4}{#1}{CM#2}{#3}
}
\newcommand\alsoimCMsym[4][\mathord]{\DeclareMathSymbol{#4}{#1}{CM#2}{#3}}

\imCMsym{cmmi}{124}{\CMjmath}
\imCMsym[\mathop]{cmsy}{113}{\CMamalg}
\imCMsym[\mathop]{cmex}{96}{\CMcoprod}
\alsoimCMsym[\mathop]{cmex}{97}{\CMbigcoprod}

\usepackage{pgfgantt}

\usepackage[style=alphabetic,hyperref=true,maxbibnames=99]{biblatex}
\DeclareFieldFormat{postnote}{#1}
\title{Noncommutative tensor triangular geometry: classification via noetherian spectra}
\author{James Rowe}
\date{August 2023}

\begin{document}

\maketitle

\begin{abstract}
Given a monoidal triangulated category $\sfT$ with noetherian spectrum, we show that there is an order preserving bijection between the collection of all Thomason subsets of the non-commutative spectrum $\Spc(\sfT)$ and the collection of all thick two-sided semiprime ideals of $\sfT$. This provides an alternative to the hypotheses of \cite{noncomttbasics}, as well as the recent approach via completely prime ideals in \cite{completelyprimeclass}. By assuming the spectrum is noetherian, we show that it is indeed a spectral space, and that it is universal among all such spaces classifying the ideals in question.  
\end{abstract}
\section{Introduction}

One of the applications of tensor triangular geometry is the classification of various flavours of subcategories via particular topological spaces. While classifications had been computed in fields such as algebraic geometry and representation theory, Balmer provided the prototypical tt-theoretic classification \cite{BaSpec}. Since then the theory of classification has expanded to include the use residue functors, categorical actions and large spectra to further increase the range of subcategories which can be classified.

This article considers a different approach, building on the work of Nakano, Vashaw and Yakimov \cite{noncomttbasics}. Instead of tensor triangulated categories, we study the theory of \emph{monoidal} triangulated categories, or \emph{mt-categories}. In particular, we no longer require the tensor product $\otimes$ to be a symmetric functor. 
The classification theorem of \cite{noncomttbasics} applies to a large collection of \emph{weak support data} but in exchange for this scale carries with it a large collection of additional hypotheses. These include defining the support over all objects in a large mt-category and requiring desirable behaviour over closed subsets of the spectrum.

Various examples of interest can be shown to satisfy these conditions \cite{noncomttbasics}. However, when comparing to the tt-theoretic classification given by Balmer, the list of additional hypotheses in the monoidal case is significant. It is not immediately clear whether or not the usual universal support always classifies any thick ideals at all. Moreover, there is something uncomfortable in appealing to \emph{localising} subcategories when only attempting to classify \emph{thick} subcategories. Attempting to extend the usual support to apply to non-compact objects in a useful way is non-trivial and it is not obvious that the usual methods can be applied to the noncommutative setting.  
Recent work by Mallick and Ray in \cite{completelyprimeclass} classifies radical thick tensor ideals of mt-categories using a point free approach, under the assumption that every prime ideal is completely prime. That is, given a prime ideal $\mcP$, if there are objects $a,b$ such that $a\otimes b \in \mcP$, then $a \in \mcP$ or $b \in \mcP$. By making this assumption about complete primeness, the authors can dispense with the other restrictions present in the classification in \cite{noncomttbasics}. This also recovers the symmetric case as for tt-categories this condition is just the usual definition of being prime.
In this article we will instead focus on a different assumption:
\begin{assump*}
Given a mt-category $\sfT$, we assume that the non-commutative spectrum $\Spc(\sfT)$ is a noetherian topological space.
\end{assump*}

Under this noetherian assumption the usual definition of support can be extended to non-compact objects \cite[Theorem B.0.4]{nakano2022spectrum}. Moreover, if the spectrum has finite Krull dimension then this extended support satisfies a variant of the faithfulness property \cite[Theorem B.0.5]{nakano2022spectrum}.

Rather than trying to further refine this extended support, we will instead focus only on compact objects and thick subcategories. We follow the original work of Balmer on the symmetric case (see for instance \cite{BaSpec}) and prove analogous statements in the monoidal setting. With this framework we obtain the following classification:

\begin{thm*}
There is an order preserving bijection between the collection of all Thomason subsets of the spectrum $\Spc(\sfT)$ and the collection of all thick two-sided semiprime ideals of $\sfT$.
\end{thm*}

We also show that if we are given a classification via some other well-behaved support datum $(X,\sigma)$ then the universal map becomes a homeomorphism.

\begin{thm*}
Suppose $(X,\sigma)$ is a support datum which classifies all thick two-sided semiprime tensor ideals of $\sfT$, and the space $X$ is noetherian and $T_0$. Then the universal map $f_\sigma: X \xrightarrow{} \Spc(\sfT)$ is a homeomorphism.
\end{thm*}

In proving these we also show that under the noetherian assumption, the noncommutative spectrum $\Spc(\sfT)$ is indeed a \emph{spectral} space.
\newline
\\
\textbf{Acknowledgements.} Many thanks to Greg Stevenson and Kent Vashaw for their insightful comments and suggestions. 

\section{Preliminaries on mt-categories}

We will start with the key definitions related to mt-categories and their spectra as introduced in \cite{noncomttbasics}.

\begin{defn}
An \emph{essentially small monoidal triangulated category} is a triple of the form $(\sfT, \otimes, \unit)$, where $\sfT$ is an essentially small triangulated category and $(\otimes, \unit)$ is a monoidal structure on $\sfT$ such that $\otimes$ is an exact functor in each variable. In other words, it is a tensor triangulated category but the tensor need not be symmetric. 

We will abbreviate ``monoidal triangulated category" to ``mt-category". 
\end{defn}

The following definition contains the different notions of prime in the noncommutative setting.
\begin{defn}\cite[1.2]{noncomttbasics}
	Let $\sfT$ be an essentially small mt-category.
	\begin{enumerate}
		\item A \emph{thick two-sided ideal} of $\sfT$ is a thick subcategory closed under left and right tensoring with arbitrary objects of $\sfT$.
		\item A \emph{prime ideal} of $\sfT$ is a proper thick ideal $\mcP$ such that $\mcI \otimes \mcJ \subseteq \mcP$ implies $\mcI \subseteq \mcP$ or $\mcJ \subseteq \mcP$ for all thick two-sided ideals $\mcI$ and $\mcJ$ or $\sfT$.  
		\item A \emph{semiprime} ideal of $\sfT$ is an intersection of prime ideals of $\sfT$.  
		\item A \emph{completely prime} ideal of $\sfT$ is a proper thick ideal $\mcP$ such that $A \otimes B \in \mcP$ implies $A \in \mcP$ or $B \in \mcP$ for all objects $A, B \in \sfT$. 
	\end{enumerate}
\end{defn}

\begin{defn}\label{noncomspecdef}\cite{noncomttbasics}
	The \emph{noncommutative Balmer spectrum} of an essentially small mt-category $\sfT$ is the topological space of prime ideals of $\sfT$. We denote the noncommutative spectrum by $\Spc(\sfT)$. The topology on $\Spc(\sfT)$ is generated by the collection of closed subsets
	
\[V(S)=\{\mcP \in \Spc(\sfT) \ \vert \ \mcP \cap S = \varnothing\}\]
for all subsets $S$ of $\sfT$. 
\end{defn}

We have the following characterisation of prime ideals in terms of objects rather than ideals:
\begin{thm}\emph{\cite[1.2.1]{noncomttbasics}}
	Let $\sfT$ be an essentially small mt-category. Then the following hold:
	\begin{enumerate}
		\item A proper thick ideal $\mcP$ of $\sfT$ is prime if and only if, given objects $A,B \in \sfT$, we have $A \otimes C \otimes B \in \mcP$, for all $C \in \sfT$ implies $A \in \mcP$ or $B \in \mcP$. 
		\item A proper thick ideal $\mcP$ of $\sfT$ is semiprime if and only if, given $A \in \sfT$, we have $A \otimes C \otimes A \in \mcP$, for all $C \in \sfT$ implies $A \in \mcP$. 
		\item The noncommutative Balmer spectrum $\Spc(\sfT)$ is always nonempty. 
	\end{enumerate}
\end{thm}

Just like in the symmetric setting, there is a universal support datum which is the noncommutative analogue to Balmer's notion of support.

\begin{defn}\cite{noncomttbasics}
	For an essentially small mt-category $\sfT$, the \emph{small noncommutative support} of an object $t$ is given by 
	\[ \supp(t)=\{\mcP \in \Spc(\sfT) \ \vert \ t \not\in \mcP\}.\]
 Note that this is just the restriction to objects of the map $V$ of Definition \ref{noncomspecdef}. 
\end{defn}

\begin{lem}\emph{\cite[4.1.2]{noncomttbasics}}\label{smallnoncomsupp}
	The small noncommutative support satisfies the following properties:
	\begin{enumerate}
		\item $\supp(0)=\varnothing$ and $\supp(\unit)=\Spc(\sfT)$.
		\item $\supp(t \oplus s) = \supp(t) \cup \supp(s),$ for all $t,s \in \sfT$.
		\item $\supp(\Sigma t) = \supp(t)$.
		\item If $t\to s \to r \to \Sigma t$ is a distinguished triangle, then $\supp(t) \subseteq \supp(s) \cup \supp(r).$
		\item $\bigcup_{r \in \sfT}\supp(t \otimes r \otimes s)=\supp(t) \cap \supp(s)$ for all $t, s \in \sfT$.
		\item For all $t \in \sfT$ the subset $\supp(t)$ is closed.
	\end{enumerate}
\end{lem}

The properties of the usual support become the definition of a \emph{support datum}:

\begin{defn}\cite[4.1.1]{noncomttbasics}\label{supportdatumdefn}
A support datum on a mt-category $\sfT$ is a pair $(X,\sigma)$ where $X$ is a topological space and $\sigma$ is an assignment $\sigma: \sfT \xrightarrow{} \mcX$ where $\mcX$ is the collection of all closed subsets of $X$, such that $\sigma$ satisfies the following additional properties:
\begin{enumerate}
    \item $\sigma(0)=\varnothing$ and $\sigma(\unit)=X$;
    \item $\sigma(a\oplus b)=\sigma(a)\cup \sigma(b)$, for all $a,b \in \sfT$;
    \item $\sigma(\Sigma a)=\sigma(a)$ for all $a \in \sfT$;
    \item If $a\xrightarrow{} b \xrightarrow{} c \xrightarrow{} \Sigma a$ is a distinguished triangle in $\sfT$, then $\sigma(a) \subseteq \sigma(b) \cup \sigma(c)$;
    \item $\bigcup_{c \in \sfT}\sigma(a\otimes c \otimes b) = \sigma(a) \cap \sigma(b)$.
\end{enumerate}
\end{defn}

Note that in \cite{noncomttbasics} it is not required that support data take values in \emph{closed subsets}. We will require that the supports be closed when we compare candidate support data to the universal support.

\begin{lem}\label{lemsuppintersection}
    Given a support datum $(X,\sigma)$ and a finite collection of objects $r_1,r_2,\dots,r_n$, we have
    \[
    \bigcap_{i=1}^{n}\sigma(r_i)=\bigcup_{c_1,c_2,\dots,c_{n-1} \in \sfT}\sigma(r_1 \otimes c_1 \otimes r_2 \otimes c_2 \otimes \cdots \otimes c_{n-1} \otimes r_n)
    \]
\end{lem}

\begin{proof}
    We proceed by induction, where the base case $\sigma(r_1)\cap\sigma(r_2)$ is satisfied by Definition \ref{supportdatumdefn}. Suppose the result holds for the $n-1$ case. Then
    \begin{align*}
       \bigcap_{i=1}^{n}\sigma(r_i)&=\bigg(\bigcap_{i=1}^{n-1}\sigma(r_i)\bigg) \cap \sigma(r_n)\\ 
       &=\bigg(\bigcup_{c_1,c_2,\dots,c_{n-2} \in \sfT}\sigma(r_1 \otimes c_1 \otimes \cdots \otimes c_{n-2} \otimes r_{n-1})\bigg)\cap\sigma(r_n)\\
       &=\bigcup_{c_1,c_2,\dots,c_{n-2} \in \sfT}\bigg(\sigma(r_1 \otimes c_1 \otimes \cdots \otimes c_{n-2} \otimes r_{n-1})\cap \sigma(r_n)\bigg)\\
       &=\bigcup_{c_1,c_2,\dots,c_{n-2} \in \sfT}\bigcup_{c_{n-1}\in\sfT}\sigma(r_1 \otimes c_1 \otimes r_2 \otimes c_2 \otimes \cdots \otimes c_{n-1} \otimes r_n)\\
       &=\bigcup_{c_1,c_2,\dots,c_{n-1} \in \sfT}\sigma(r_1 \otimes c_1 \otimes r_2 \otimes c_2 \otimes \cdots \otimes c_{n-1} \otimes r_n)
    \end{align*}
    as required.
\end{proof}

\section{The noncommutative spectrum is spectral}

We wish to show that the spectrum of a monoidal triangulated category is a spectral space. Let us make it clear what this actually means:

\begin{defn}
Let $X$ be a topological space and let $K(X)$ denote the set of all quasi-compact open subsets of $X$. The topological space $X$ is \emph{spectral} if it satisfies all of the following conditions:
\begin{enumerate}
    \item $X$ is quasi-compact and $T_0$. By $T_0$ we mean that given any two points $x,y \in X$ there is an open subset of $X$ containing one of these points, but not the other.
    \item $K(X)$ is a basis of open subsets for $X$.
    \item $K(X)$ is closed under finite intersections.
    \item $X$ is a sober space. That is, every irreducible closed subset of $X$ has a necessarily unique generic point.
\end{enumerate}

\end{defn}

We will rely on assuming that the spectrum is a noetherian topological space:

\begin{defn}
A topological space $X$ is noetherian if any of the following equivalent conditions hold:
\begin{enumerate}
    \item $X$ satisifies the descending chain condition for closed subsets. That is, for any sequence 
    \[
    Y_1 \supseteq Y_2 \supseteq \cdots
    \]
    of closed subsets $Y_i$ of $X$, there exists an integer $m$ such that for all integers $n\geq m$ we have $Y_m=Y_n$.
    \item Every subspace of $X$ is quasi-compact.
    \item Every open subset of $X$ is quasi-compact.
\end{enumerate}
\end{defn}

Let $\sfT$ be an essentially small monoidal triangulated category with spectrum $\Spc(\sfT)$. We aim to show that $\Spc(\sfT)$ is a spectral topological space.
As seen in Definition \ref{noncomspecdef}, a basis of closed sets is given by $
\{V(\mcS) \ \vert \ \mcS \subseteq \sfT\}
$, where
\[
V(\mcS)=\{\mcP \in \Spc(\sfT) \ \vert \ \mcP \cap \mcS = \varnothing\}.
\]
The corresponding basis of open sets is $
\{U(\mcS) \ \vert \ \mcS \subseteq \sfT\}
$ where
\[
U(\mcS)=\{\mcP \in \Spc(\sfT) \ \vert \ \mcP \cap \mcS \neq \varnothing\}.
\]
Immediately we have $V(\mcS)=\bigcap_{s \in \mcS}V(s)$ and $U(\mcS)=\bigcup_{s \in \mcS}U(s)$. Therefore the sets of the form $V(s)=\supp(s)$ form a basis of closed sets for the topology, while the sets of the form $U(s)$ form a basis of open sets for the topology.

Throughout given a collection of objects $\mcE$ define $\supp(\mcE)=\bigcup_{s \in \mcE}\supp(s)$.

\begin{lem}\label{basicclosures}
Let $Y \subseteq \Spc(\sfT)$. Then the closure of $Y$ is given by
\[
\overline{Y}=\bigcap_{Y \subseteq \supp(t)}\supp(t).
\]
\end{lem}

\begin{proof}
Immediate from the fact that the $\supp(s)$ are a basis of closed sets for the topology on $\Spc(\sfT)$. 
\end{proof}

\begin{prop}\label{mttzero}
For any point $\mcP \in \Spc(\sfT)$, the closure of $\mcP$ is given by
\[
\overline{\{\mcP\}}=\{\mcQ \in \Spc(\sfT) \ \vert \ \mcQ \subseteq \mcP\}.
\]
In particular, if $\overline{\{\mcP_1\}}=\overline{\{\mcP_2\}}$ then $\mcP_1=\mcP_2$. That is, the space $\Spc(\sfT)$ is $T_0$.
\end{prop}

\begin{proof}
The proof is identical to \cite[2.9]{BaSpec}. Fix a prime ideal $\mcP$. Consider the set $\mcS_0=\sfT\setminus \mcP$ and the associated basic closed subset $V(\mcS_0)=\{\mcQ \ \vert \mcQ \cap \mcS_0=\varnothing\}$. Clearly $\mcP \in V(\mcS_0)$. If there is a subset $\mcS \subseteq \sfT$ such that $\mcP \in V(\mcS)$, then $\mcS \subseteq \mcS_0$ and so $V(\mcS_0) \subseteq V(\mcS)$. Therefore $V(\mcS_0)$ is the smallest closed subset containing $\mcP$ and is the closure of $\mcP$. We have
\[
\overline{\{\mcP\}}=V(\mcS_0)=\{\mcQ \in \Spc(\sfT) \ \vert \ \mcQ \subseteq \mcP\}.
\]
The fact that $\Spc(\sfT)$ is $T_0$ follows immediately.
\end{proof}

We will make use of the following theorem from \cite{noncomttbasics}, which is the non-symmetric version of \cite[2.2]{BaSpec}.

\begin{thm}\emph{\cite[3.2.3]{noncomttbasics}}\label{monoidalavoidance}
Suppose $\mcM$ is a multiplicative subset of $\sfT$ and suppose $\mcI$ is a proper thick two-sided tensor ideal of $\sfT$ such that $\mcI \cap \mcM = \varnothing$. The set
\[
X(\mcM,\mcI)=\{\mcJ\text{ a thick two-sided tensor ideal of }\sfT \ \vert \ \mcI \subseteq \mcJ,\ \mcJ \cap \mcM = \varnothing\}
\]
has a maximal element, and moreover this maximal element is prime.
\end{thm}

We now bring in our noetherian assumption.

\begin{assump*}
From now on, we shall assume that the spectrum $\Spc(\sfT)$ is a noetherian topological space.
\end{assump*}

\begin{rem}\label{noetherianconsequence}
By assuming that $\Spc(\sfT)$ is noetherian, it immediately follows that $\Spc(\sfT)$ is quasi-compact, as are all of the basic open subsets $U(\mcS)$, including those of the form $U(t)$ for all objects $t \in \sfT$. 
\end{rem}

\begin{rem}
    As noted in the introduction, there are other possible assumptions on the market. The alternative assumption that all prime ideals are completely prime is used in \cite{completelyprimeclass} to classify ideals. The completely prime assumption may hold when the noetherian assumption fails, such as in the symmetric case of the stable homotopy category of spectra. However, even in those cases where one assumption implies the other, it may be easier to verify that the spectrum is noetherian than it is to check that all primes are completely prime. The techniques of both approaches vary significantly, with the proof in \cite{completelyprimeclass} using a point-free approach. 
\end{rem}

\begin{prop}\label{mtgenericpoints}
Non-empty irreducible subsets of $\Spc(\sfT)$ have unique generic points. Indeed for a non-empty closed subset $Z \subseteq \Spc(\sfT)$ the following are equivalent:
\begin{enumerate}
    \item $Z$ is irreducible.
    \item For all $t,s \in \sfT$, if $U(t\oplus s)\cap Z = \varnothing$, then $U(t)\cap Z= \varnothing$ or $U(s) \cap Z =\varnothing$.
    \item The collection $\mcP=\{t \in \sfT \ \vert \ U(t)\cap Z \neq \varnothing\}$ is a thick prime $\otimes$-ideal.
\end{enumerate}
Moreover, when these conditions hold $Z=\overline{\{\mcP\}}$.
\end{prop}

\begin{proof}
The proof is very similar to \cite[2.18]{BaSpec}, although some extra care is needed when proving things are prime. We have already seen that $\Spc(\sfT)$ is $T_0$ and so uniqueness of generic points is immediate.\\
$(1)\implies (2)$:\ $Z$ irreducible means that for any open subsets $U_1, U_2 \in \Spc(\sfT)$, if $Z \cap U_1 \cap U_2 = \varnothing$ then $Z\cap U_1 =\varnothing$ or $Z\cap U_2 =\varnothing$. This gives $(2)$, since $U(t\oplus s)=U(t)\cap U(s)$.\\
$(2)\implies (3)$:\ This will be slightly more involved then the proof in \cite{BaSpec}. The assumption $(2)$ gives $t,s \in \mcP$ implies $t \oplus s \in \mcP$. Using this, we see
that if $t,s \in \mcP$ and $t\xrightarrow{} s \xrightarrow{} r \xrightarrow{} \Sigma t$ is a distinguished triangle, then $r \in \thick^\otimes(t\oplus s)$, hence $U(t\oplus s)\subseteq U(r)$ and since $U(t\oplus s)\cap Z \neq \varnothing$, we get $U(r)\cap Z \neq \varnothing$, and so $r \in \mcP$. The fact that $\mcP$ is closed under summands is immediate as $U(t\oplus s) \cap Z \neq \varnothing$ implies $(U(t)\cap U(s))\cap Z \neq \varnothing$ and therefore $U(t)\cap Z \neq \varnothing$ and $U(s) \cap Z \neq \varnothing$. Therefore both $t$ and $s$ are objects in $\mcP$. It remains to show that $\mcP$ is a two-sided ideal, and that it is prime. Fix $t \in \mcP$ and $x \in \sfT$. We have
\begin{align*}
    \varnothing &\neq Z \cap U(t)\\
    &\subseteq Z \cap (U(t) \cup U(x))\\
    &= Z \cap (\bigcap_{s \in \sfT}U(t\otimes s \otimes x)) \text{ by Lemma \ref{smallnoncomsupp}}\\
    &\subseteq Z \cap U(t \otimes x).
\end{align*}
Therefore $t\otimes x \in \mcP$. An almost identical argument shows that $x \otimes t\in \mcP$ and so $\mcP$ is indeed a two-sided ideal. Now we deal with primeness. Let $\mcI,\mcJ$ be thick $\otimes$-ideals such that $\mcI \otimes \mcJ \subseteq \mcP$. In particular, for all $i \in \mcI$ and $j \in \mcJ$ we have $U(i\otimes j) \cap Z \neq \varnothing$. Now,
\[
\bigcup_{i\in\mcI,j\in\mcJ}\supp(i\otimes j)=\bigcup_{i \in \mcI}\supp(i) \cap \bigcup_{j \in \mcJ}\supp(j),
\]
see for example \cite[4.4.2]{noncomttbasics}.
Therefore 
\[
\bigcap_{i\in \mcI,j\in\mcJ}U(i\otimes j) = \bigcap_{i \in \mcI}U(i) \cup \bigcap_{j \in \mcJ}U(j).
\]
As we assumed $\mcI \otimes \mcJ \subseteq \mcP$ we must have $\mcP \in U(i\otimes j)$ for all $i$ and $j$. Therefore $\mcP \in U(i)$ for all $i$ or $\mcP \in U(j)$ for all $j$, which is equivalent to asking that $\mcI \subseteq \mcP$ or $\mcJ \subseteq \mcP$, and so we conclude that $\mcP$ is indeed prime.\\
$(3)\implies(1)$:\ We prove that $Z = \overline{\{P\}}$ which proves $(1)$ and and the final statment of the proposition.
Let $\mcQ \in Z$. For $a \in \mcQ$, we have $\mcQ \in U(a) \cap Z \neq \varnothing$, hence $a \in \mcP$. We have proved
$\mcQ \subseteq \mcP$, that is, $\mcQ \in \overline{\{P\}}$ by Proposition \ref{mttzero} for any $\mcQ \in Z$. So, we have $Z \subseteq \overline{\{P\}}$. Conversely,
it suffices to prove $P \in Z$. To see this, let $s \in \sfT$ be an object such that
$Z \subseteq \supp(s)$. Such objects exist by Lemma \ref{basicclosures}. Then $U(s) \cap Z = \varnothing$ which means $s \not\in \mcP$, or equivalently $\mcP \in \supp(s)$.
Therefore,
\[\mcP \in \bigcap_{Z \subseteq \supp(s)}\supp(s)=\overline{Z}=Z,\] by Lemma \ref{basicclosures}.
\end{proof}

Putting together the results so far, we obtain the following theorem:

\begin{thm}
If $\Spc(\sfT)$ is noetherian, then it is spectral.
\end{thm}

\begin{proof}
We verify the conditions required to be spectral.
\begin{enumerate}
    \item $\Spc(\sfT)$ is quasi-compact by Remark \ref{noetherianconsequence}. The space is $T_0$ by Proposition \ref{mttzero}.
    \item Under the noetherian assumption, every open subset is quasi-compact and so it is immediate that the collection $K(\Spc(\sfT))$ of all quasi-compact open subsets is a basis for $\Spc(\sfT)$. 
    \item Given quasi-compact basic opens of the form $U(t)$ and $U(s)$ we have $U(s) \cap U(t) = U(s\oplus t)$. For a quasi-compact basic open of the form $U(\mcS)=\bigcup_{s \in \mcS}U(s)$ by quasi-compactness there exists a finite subset $\mcS^\prime \subseteq \mcS$ such that $U(\mcS)=U(\mcS^\prime)=\bigcup_{s \in \mcS^\prime}U(s)$. Given another such quasi-compact basic open $U(\mcT)$, with finite refinement $\mcT^\prime$, we obtain
    \begin{align*}
        U(\mcS)\cap U(\mcT) &= \bigcup_{s \in \mcS^\prime}U(s) \cap \bigcup_{t \in \mcT^\prime}U(t)\\
        &=\bigcup_{s \in \mcS^\prime}\bigcup_{t \in \mcT^\prime}(U(s) \cap U(t))\\
        &=\bigcup_{s \in \mcS^\prime}\bigcup_{t \in \mcT^\prime}U(s\oplus t).
    \end{align*}
    As $\mcS^\prime$ is a finite set, and $\mcT^\prime$ is a finite refinement, both unions are finite. As $U(\mcS)$ and $U(\mcT)$ are quasi-compact, it follows that the intersection is quasi-compact. 
    \item By Proposition \ref{mtgenericpoints}, every non-empty irreducible subsets of $\Spc(\sfT)$ has a unique generic point. In other words, $\Spc(\sfT)$ is sober.
\end{enumerate}
As all of the conditions are satisfied, we conclude that the spectrum $\Spc(\sfT)$ is a spectral space.
\end{proof}

\begin{rem}
    There are other conditions on the mt-category $\sfT$ which can lead to $\Spc(\sfT)$ being spectral. For example, \cite[6.7]{BKS} proves that the spectrum is spectral under the assumption that the tensor product is symmetric, or that the mt-category $\sfT$ has a generator. 
\end{rem}

\section{Classifying thick two-sided ideals}
The objective of this section is to classify all \emph{semiprime} thick tensor ideals of a mt-category $\sfT$ in terms of Thomason subsets of the spectrum $\Spc(\sfT)$. We do this under the hypothesis that $\Spc(\sfT)$ is noetherian, which we assume throughout. If the mt-category $\sfT$ is rigid, then the classification actually covers \emph{all} thick two-sided tensor ideals, recovering Balmer's original classification \cite[4.10]{BaSpec} in the symmetric case. 

\begin{lem}\label{collectionsupp}
Given a collection of objects $\mcE \in \sfT$ there is an equality
\[
\supp(\mcE)=\{\mcP \in \Spc(\sfT)\ \vert \ \mcE \not\subseteq \mcP\}.
\]
\end{lem}

\begin{proof}
The proof is identical to \cite[4.6]{BaSpec}. We have $\mcP \in \supp(\mcE)$ if and only if there exists an object $a \in \mcE$ such that $\mcP \in
\supp(a)$ which means $a \not\in \mcP$, by definition of the support.
\end{proof}

\begin{defn}
Let $\mcJ$ be a thick tensor ideal of $\sfT$. We denote by $\sqrt{\mcJ}$ the semiprime ideal
\[
\sqrt{\mcJ}=\bigcap_{\mcJ \subseteq \mcP \in \Spc(\sfT)}\mcP.
\]
\end{defn}

\begin{defn}
Let $Y \subseteq \Spc(\sfT)$ be a subset. Define the full subcategory $\sfT_Y$ by
\[
\sfT_Y=\{t \in \sfT \ \vert \ \supp(t)\subseteq Y\}.
\]
\end{defn}

\begin{lem}\label{subcatsupportsdirect}
The subcategory $\sfT_Y$ is a thick two-sided tensor ideal.
\end{lem}

\begin{proof}
The statement is similar to \cite[6.1.1]{noncomttbasics}. The fact that $I_Y$ is a thick subcategory follows immediately from the usual properties of support. Now let $s \in \sfT_Y$ and $t \in \sfT$. Then
\begin{align*}
    \supp(s\otimes t)&=\supp(s\otimes\unit\otimes t)\\
    &\subseteq \bigcup_{c \in \sfT}\supp(s\otimes c\otimes t)\\
    &=\supp(s) \cap \supp(t)\\
    &\subseteq Y.
\end{align*}
Therefore $\supp(s\otimes t)\subseteq Y$ and $s\otimes t \in \sfT_Y$. That is, $\sfT_Y$ is a right ideal. A similar argument shows that $\sfT_Y$ is a left ideal.
\end{proof}

\begin{rem}
Note that if $(X,\sigma)$ is a support datum on $\sfT$ then above lemma can be adjusted to show that the full subcategory $\{t \in \sfT \ \vert \ \sigma(t) \subseteq Y\}$ is a thick two-sided ideal.
\end{rem}

\begin{lem}\label{subcatsupportssemiprime}
There is an equality 
\[
\sfT_Y=\bigcap_{\mcP \not\in Y}\mcP\text{ where }\mcP \in \Spc(\sfT).
\]
\end{lem}

\begin{proof}
The proof is identical to \cite[4.8]{BaSpec}. For an object $t \in \sfT$, we have $t \in \sfT_Y$ if and only if $\supp(t)\subseteq Y$. Therefore, for all $\mcP \in \Spc(\sfT)\setminus Y$, $t \in \sfT_Y$ if and only if $\mcP \not\in \supp(t)$ and $t \not\in\mcP$. Hence $t \in \bigcap_{\mcP \not\in Y}\mcP$ and the conclusion holds.
\end{proof}

\begin{prop}\label{tysupporradical}
Let $\mcJ$ be a thick tensor ideal of $\sfT$. Then
\[
\sfT_{\supp(\mcJ)}=\sqrt{\mcJ}.
\]
\end{prop}

\begin{proof}
Identical to \cite[4.9]{BaSpec}. By Lemma \ref{subcatsupportssemiprime}, we have
\[\sfT_{\supp(\mcJ)}=\bigcap_{\mcP \not \in \supp(\mcJ)}\mcP\]
Applying Lemma \ref{collectionsupp}, gives $\supp(\mcJ)=\{\mcQ \in \Spc(\sfT) \ \vert \ \mcJ \not\subseteq \mcQ\}$ and so
\[\sfT_{\supp(\mcJ)}=\bigcap_{\mcP \not\in \{\mcQ \in \Spc(\sfT) \ \vert \ \mcJ \not\subseteq \mcQ\}}\mcP\]
The result then immediately follows from the definition of $\sqrt{\mcJ}$.
\end{proof}

\begin{notation}
We denote by $\mcT$ the collection of all Thomason subsets of $\Spc(\sfT)$. Recall that a subset $Y \subseteq \Spc(\sfT)$ is Thomason if $Y=\bigcup Y_i$ such that each $Y_i$ is closed and the open complement $\Spc(\sfT)\setminus Y_i$ is quasi-compact.
We denote by $\mcS$ the collection of all semiprime ideals of $\sfT$.
\end{notation}

\begin{thm}\label{mtclassification}
There is an order preserving bijection $\mcT \xrightarrow{\sim} \mcS$ given by
\[
Y \xrightarrow{} \sfT_Y
\]
whose inverse is
\[
\mcJ \xrightarrow{} \supp(\mcJ).
\]
\end{thm}

\begin{proof}
The first map is well-defined as $\sfT_Y$ is semiprime by Lemma \ref{subcatsupportssemiprime}. The second map is well defined as the complement of an object's support is quasi-compact under the noetherian assumption. Both maps are clearly inclusion preserving. It remains to show the maps are mutually inverse. Given a semiprime ideal $\mcJ$ the composite $\sfT_{\supp(\mcJ)}$ is equal to $\sqrt{\mcJ}$ by Proposition \ref{tysupporradical}. By assumption $\mcJ$ is semiprime and so $\mcJ=\sqrt{\mcJ}$, and so the composite 
\[\mcJ \xrightarrow{} \supp(\mcJ)\xrightarrow{} \sfT_{\supp(\mcJ)}\]
is the identity. 
Now let $Y$ be a Thomason subset of $\Spc(\sfT)$. It remains to show that the composition
\[
Y \xrightarrow{} \sfT_Y \xrightarrow{} \supp(\sfT_Y)
\]
is the identity. For an object $t \in \sfT$ we have by definition $t \in \sfT_Y$ if and only if $\supp(t)\subseteq Y$. Therefore
\[
\supp(\sfT_Y)=\bigcup_{t \in \sfT_Y}\supp(t)\subseteq Y.
\]
Now we need to show that $Y \subseteq \supp(\sfT_Y)$. That is, for each prime ideal $\mcP \in Y$ we must find a compact object $x$ such that $\mcP \in \supp(x)$ and $\supp(x) \subseteq Y$. As $Y$ is a Thomason subset of $\Spc(\sfT)$, there exist closed subsets $Y_i$ such that $Y=\bigcup Y_i$ and the complement of each $Y_i$ is a quasi-compact open subset $U_i$. Fix a prime $\mcP \in Y$. Then there exists an index $i$ such that
\[
\mcP \in Y_i = \Spc(\sfT) \setminus U_i.
\]
By assumption $\Spc(\sfT)$ is noetherian, so there exists a finite collection of objects $\{r_1, \dots, r_n\}$ such that $U_i=\bigcup_{j=1}^{n} U(r_j)$. Therefore
\begin{align*}
    Y_i &= \Spc(\sfT)\setminus U_i\\
    &=\Spc(\sfT)\setminus (\bigcup_{j=1}^n U(r_j))\\
    &=\bigcap_{j=1}^n (\Spc(\sfT)\setminus U(r_j))\\
    &=\bigcap_{j=1}^n \supp(r_j).\\
\end{align*}

By Lemma \ref{lemsuppintersection}, there exists compact objects $c_1,\dots, c_{n-1}$ such that for 
\[x=r_1 \otimes c_1 \otimes \cdots \otimes c_{n-1} \otimes r_n
\]
we have $\mcP \in \supp(x)$. Moreover, 
\[
\supp(x) \subseteq \bigcap_{j=1}^{n}\supp(r_j)=Y_i \subseteq Y
\]
and so $x \in \sfT_Y$, completing the proof. 
\end{proof}

\begin{prop}\emph{\cite[4.1.1]{noncomintersections}}
Suppose $\sfT$ is rigid, so that every object is either left or right dualisable. Then every thick two-sided tensor ideal is semiprime.
\end{prop}

\begin{cor}
    Let $\sfT$ be a rigid mt-category with noetherian spectrum $\Spc(\sfT)$. Then the order preserving bijection of Theorem \ref{mtclassification} classifies all thick two-sided tensor ideals of $\sfT$.
\end{cor}

\section{Classifying support data and the universal map}

Working under the noetherian assumption, we can now investigate the universality of the spectrum with respect to classifying support data. This section provides the monoidal analogue of \cite[5.2]{BaSpec}. 

\begin{defn}
A subset $Y\subset X$ of a topological space $X$ is \emph{specialisation closed} if it is the union of closed sets, or equivalently if $y \in Y$ implies $\overline{\{y\}}\subseteq Y$. Given a topological space $X$ we denote by $\mcX_{\spe}$ the collection of all specialisation closed subsets of $X$.
\end{defn}

Recall that we denote the collection of all thick semiprime ideals of $\sfT$ by $\mcS$.

\begin{defn}
Let $(X,\sigma)$ be a support datum on $\sfT$. We say that $(X,\sigma)$ is a \emph{classifying support datum} if the following two conditions hold:
\begin{enumerate}
    \item The space $X$ is noetherian and spectral.
    \item We have a bijection $\Theta:\mcX_\spe \xrightarrow{} \mcS$ defined by $\Theta(Y)=\{t \in \sfT \ \vert \ \sigma(t)\subseteq Y\}$ with inverse $\Theta^{-1}(\mcJ)=\sigma(\mcJ)=\bigcup_{j \in \mcJ}\sigma(j)$.
\end{enumerate}
\end{defn}

\begin{lem}\label{closedaresupports}
Suppose $(X,\sigma)$ is a classifying support datum on $\sfT$. Then every closed subset $Z\subseteq X$ is of the form $Z=\sigma(t)$ for some object $t \in \sfT$.
\end{lem}

\begin{proof}
This is the first claim of \cite[5.2]{BaSpec}. The proof is identical and included for completeness. As $X$ is noetherian, every closed subset has a finite number of irreducible components. Since $\sigma(t_1) \cup \sigma(t_2) \cup \cdots \cup \sigma(t_n)=\sigma(t_1\oplus t_2 \oplus \cdots \oplus t_n)$ for any finite collection of objects in $\sfT$, it therefore suffices to prove the lemma for closed sets of the form $Z=\overline{\{x\}}$ for some $x \in X$. As $(X,\sigma)$ is classifying we have 
\[
\overline{\{x\}}=Z=\Theta^{-1}\Theta(Z)=\bigcup_{t \in \Theta(Z)}\sigma(t).
\]
Therefore there exists $t \in \sfT$ such that $x \in \sigma(t)\subseteq Z$. Hence
\[
\overline{\{x\}}\subseteq \sigma(t) \subseteq Z = \overline{\{x\}},
\]
proving the lemma.
\end{proof}

\begin{cor}
Every open subset of $\Spc(\sfT)$ is of the form $U(t)$ for some object $t \in \sfT$.
\end{cor}

\begin{proof}
For $\Spc(\sfT)$ noetherian, Theorem \ref{mtclassification} tells us that $(\Spc(\sfT),\supp)$ is a classifying support datum on $\sfT$. Therefore given an open subset $U\subseteq \Spc(\sfT)$, Lemma \ref{closedaresupports} tells us that $\Spc(\sfT)\setminus U=\supp(t)$ for some object $t \in \sfT$. Then $U=\Spc(\sfT)\setminus \supp(t) = U(t)$. Under the noetherian assumption $U(t)$ is quasi-compact.
\end{proof}

The following theorem describes the universal property of the support datum $(\Spc(\sfT),\supp)$:

\begin{thm}\emph{\cite[4.2.2]{noncomttbasics}}\label{mtuniversal} Let $(X,\sigma)$ be a support datum on $\sfT$ such that $\sigma(t)$ is closed for every object $t \in \sfT$. Then there is a unique continuous map $f_\sigma:X\xrightarrow{} \Spc(\sfT)$ satisfying $\sigma(t)=f^{-1}_\sigma(\supp(t))$ for all $t \in \sfT$. In other words, $(\Spc(\sfT),\supp)$ is the final among all such support data. The map $f_\sigma$ is given by
\[
f_\sigma(x)=\{t \in \sfT \ \vert \ x \not\in \sigma(t)\}.
\]
\end{thm}

\begin{prop}\label{fsigmainjective}
If $(X,\sigma)$ is a classifying support datum on $\sfT$, then the universal map $f_\sigma: X \xrightarrow{} \Spc(\sfT)$ is injective.
\end{prop}

\begin{proof}
This is the same as the proof of injectivity in \cite[5.2]{BaSpec} and is included for completeness. For $x \in X$ define $Y(x)=\{y \in X \ \vert \ x \not\in \overline{\{y\}}\}$. Clearly $Y(x)$ is specialisation closed. Fix an object $t \in \sfT$. We will show that $\sigma(t) \subseteq Y(x)$ if and only if $x \not\in \sigma(t)$. Since $x \not\in Y(x)$, if $\sigma(t)\subseteq Y(x)$ then $x\not\in \sigma(t)$. Conversely, as $\sigma(t)$ is specialisation closed, if $x \not\in \sigma(t)$ we have $x \not\in\overline{\{y\}}$ for all $y \in \sigma(t)$ and so by definition $\sigma(t)\subseteq Y(x)$. Therefore
\[
\Theta(Y(x))=\{t \in \sfT \ \vert \ \sigma(t)\subseteq Y(x)\} = \{t \in \sfT \ \vert \ x \not\in \sigma(t)\} = f_\sigma(x).
\]
As $(X,\sigma)$ is classifying, if $f_\sigma(x_1)=f_\sigma(x_2)$ then $Y(x_1)=Y(x_2)$ and $\overline{\{x_1\}}=\overline{\{x_2\}}$. The space $X$ must be $T_0$ as $(X,\sigma)$ is classifying, so $\overline{\{x_1\}}=\overline{\{x_2\}}$ implies $x_1=x_2$ and the map $f_\sigma$ is injective.
\end{proof}

\begin{prop}\label{fsigmasurjective}
    If $(X,\sigma)$ is a classifying support datum on $\sfT$, then the universal map $f_\sigma: X \to \Spc(\sfT)$ is surjective.
\end{prop}

\begin{proof}
    Fix a prime $\mcP \in \Spc(\sfT)$. As $(X,\sigma)$ is classifying, there exists a specialisation closed subset $Y\subset X$ such that $\mcP=\Theta(Y)$. As $\mcP$ is proper, the set $X\setminus Y$ is nonempty. Let $x,y \in X \setminus Y$. By Lemma \ref{closedaresupports}, there exist objects $s,t \in \sfT$ such that $\overline{\{x\}}=\sigma(s)$ and $\overline{\{y\}}=\sigma(t)$. Let $\mcI$ and $\mcJ$ denote the thick two-sided ideal generated by $s$ and $t$ respectively. By \cite[4.3.2]{noncomttbasics}, we have
    \begin{align*}
        \overline{\{x\}}&=\sigma(s)=\sigma(\mcI)\\
        \overline{\{y\}}&=\sigma(t)=\sigma(\mcJ)
    \end{align*}\
    and so neither $\mcI$ nor $\mcJ$ are contained in $\mcP$. As $\mcP$ is prime, $\mcI \otimes \mcJ \not\subseteq \mcP$ and so $\sigma(\mcI \otimes \mcJ) \not\subseteq Y$. Therefore there exist a point $z \in X \setminus Y$ such that 
    \[z \in \sigma(\mcI \otimes \mcJ)= \overline{\{x\}}\cap \overline{\{y\}} \]
    and hence
    \[\overline{\{z\}}\subseteq\overline{\{x\}}\cap\overline{\{y\}}\]
    As in \cite[5.2]{BaSpec}, as the spectrum is noetherian, the non-empty family of sets
    \[\{\overline{\{x\}}\ \vert \ x \in X\setminus Y\}\]
    admits a minimal element which must be the lower bound for inclusion. That is, there exists $x \in X\setminus Y$ such that, for all $y \in X \setminus Y$ we have $x \in \overline{\{y\}}$. Hence 
    \[X \setminus Y \subseteq \{y \in X \ \vert \ x \in \overline{\{y\}}\} \]
    The reverse inclusion holds because $x \not\in Y$, and $Y$ is specialisation closed. Therefore
     \[X \setminus Y = \{y \in X \ \vert \ x \in \overline{\{y\}}\} \]
     and so
     \[Y=\{y \in X \ \vert \ x \not\in \overline{\{y\}}\}=Y(x)\]
     Hence
     \[\mcP=\Theta(Y)=\Theta(Y(x))=f_\sigma(x)\]
     where the final equality is demonstrated in the proof of Proposition \ref{fsigmainjective}. We conclude that $f_\sigma$ is surjective.
\end{proof}

\begin{thm}
Let $(X,\sigma)$ be a classifying support datum on $\sfT$. Then the universal map $f_\sigma: X \xrightarrow{} \Spc(\sfT)$ is a homeomorphism.
\end{thm}

\begin{proof}
By Proposition \ref{fsigmainjective} the universal continuous map $f_\sigma$ is injective, and by Proposition \ref{fsigmasurjective} the map is surjective. Therefore $f_\sigma$ is bijective. By Theorem \ref{mtuniversal} we have $\sigma(t)=f_\sigma^{-1}(\supp(t))$ for all $t \in \sfT$. Therefore $f_\sigma(\sigma(t))=\supp(t)$ and so $f_\sigma$ is a closed map as by Lemma \ref{closedaresupports} every closed subset of $X$ is of the form $\sigma(t)$ for some $t \in \sfT$. We therefore conclude that $f_\sigma$ is a homeomorphism.
\end{proof}

\printbibliography

\end{document}